\input ./style/arxiv-general.cfg
\documentclass[aos,MSNbibl,seceqn,dvips]{arximspdf}
\makeatletter
   \@ifpackageloaded{graphicx}{}{\usepackage{graphicx}}
\makeatother

\doi{10.1214/15-AOS1339}
\volume{43}
\issue{5}
\pubyear{2015}
\firstpage{2198}
\lastpage{2224}
\docsubty{FLA}

\makeatletter
\newcommand{\lleft}{\left}
\newcommand{\rright}{\right}
\newcommand{\rrVert}{\Vert}
\newcommand{\llVert}{\Vert}
\newcommand{\eqref}[1]{(\ref{#1})}
\newtheorem{theorem}{Theorem}[section]
\newtheorem{lemma}[theorem]{Lemma}
\newtheorem{proposition}[theorem]{Proposition}
\newtheorem{corollary}[theorem]{Corollary}
\newproclaim{remark}[theorem]{Remark}
\newproclaim{definition}[theorem]{Definition}
\renewcommand{\vec}[1]{\mathbf{#1}}
\newcommand{\vecc}[1]{\bolds{#1}}
\newcommand{\R}{\mathbb{R}}

\newcommand{\Var}{\operatorname{Var}}
\newcommand{\Ker}{\operatorname{Ker}}

\newcommand{\range}{\operatorname{range}}
\newcommand{\Wdelta}{\mathcal{W}_\Delta}
\newcommand{\WN}{\mathcal{W}_N}
\newcommand{\W}{\mathcal{W}}
\makeatother

\begin{document}
\begin{frontmatter}

\title{Computing exact $D$-optimal designs by mixed integer
second-order cone programming}
\runtitle{Computing exact optimal designs by MISOCP}

\begin{aug}
\author[A]{\fnms{Guillaume}~\snm{Sagnol}\corref{}\ead[label=e1]{sagnol@zib.de}}
\and
\author[B]{\fnms{Radoslav}~\snm{Harman}\thanksref{F2}\ead[label=e2]{harman@fmph.uniba.sk}}
\runauthor{G. Sagnol and R. Harman}
\affiliation{Zuse Institut Berlin and Comenius
University, Bratislava}
\address[A]{Department Optimization\\
Zuse Institut Berlin\\
Takustr. 7\\
14195 Berlin\\
Germany\\
\printead{e1}}
\address[B]{Faculty of Mathematics,\\
\quad Physics and Informatics\\
Comenius University\\
Mlynsk\'{a} dolina\\
84248 Bratislava\\
Slovakia\\
\printead{e2}}
\end{aug}
\thankstext{F2}{Supported by the VEGA 1/0163/13 grant of the Slovak
Scientific Grant Agency.}

%
\received{\smonth{12} \syear{2013}}
%
\revised{\smonth{1} \syear{2015}}

%
\begin{abstract}
Let the design of an
experiment be represented by an $s$-\break dimensional vector
$\vec{w}$ of weights with nonnegative components. Let the quality of
$\vec{w}$ for the estimation of the parameters of the statistical model
be measured by the criterion of $D$-optimality, defined as the $m$th root
of the determinant of the information matrix $M(\vec{w})=\sum_{i=1}^s
w_iA_iA_i^T$, where $A_i, i=1,\ldots,s$ are known matrices with $m$ rows.

In this paper, we show that the criterion of $D$-optimality is second-order
cone representable. As a result, the method of second-order cone
programming can be used to compute an approximate $D$-optimal design with
any system of linear constraints on the vector of weights. More
importantly, the proposed characterization allows us to compute an
\emph{exact} $D$-optimal design, which is possible thanks to high-quality
branch-and-cut solvers specialized to solve mixed integer\break second-order cone
programming problems.
Our results extend to the case of the criterion of $D_K$-optimality,
which measures the quality of $\vec{w}$ for the estimation
of a linear parameter subsystem defined by a full-rank coefficient
matrix $K$.

We prove that some other widely used criteria are also second-order
cone representable,
for instance, the criteria of $A$-, $A_K$-, $G$- and $I$-optimality.


We present several numerical examples demonstrating the efficiency and
general applicability of the proposed method. We show that in many
cases the mixed
integer second-order cone programming approach allows us to find a~provably
optimal exact design, while the standard heuristics systematically miss the
optimum.
\end{abstract}

%
\begin{keyword}[class=AMS]
\kwd[Primary ]{62K05}
\kwd[; secondary ]{65K05}
\end{keyword}
\begin{keyword}
\kwd{Optimal experimental design}
\kwd{exact optimal designs}
\kwd{second-order cone programming}
\kwd{mixed integer programming}
\kwd{$D$-criterion}
\end{keyword}
\end{frontmatter}

\section{Introduction}\label{sec1}

Consider an optimal experimental design problem of the form
\begin{equation}
\label{theBasicProblem} \max_{\vec{w} \in\mathcal{W}} \Phi \Biggl( \sum
_{i=1}^s w_iA_iA_i^T
\Biggr),
\end{equation}
where $\Phi$ is a criterion mapping the space $\mathbb{S}^+_m$ of $m
\times m$ positive semidefinite matrices over the set $\mathbb{R}_+:=[0,\infty)$.
In \eqref{theBasicProblem},
$A_i \in\mathbb{R}^{m \times\ell_i}$,
$i=1,\ldots,s$ are known matrices, and $\mathcal{W}$ is a compact subset of
$\mathbb{R}^s_+$ representing the set of all permissible designs.

Problem~\eqref{theBasicProblem} arises in linear regression models with
a design space $\mathcal{X} \equiv[s]:=\{1,\ldots,s\}$, independent trials
and a vector $\bolds{\theta} \in\mathbb{R}^m$ of unknown parameters,
provided that
the trial in the $i$th design point results in an $\ell_i$-dimensional
response $\vec{y}_i$, satisfying $E(\vec{y}_i)=A_i^T\vecc{\theta}$ and
$\Var(\vec{y}_i)=\sigma^2 \vec{I}_{\ell_i}$, where $\vec{I}_k$ is the
$k \times k$-identity matrix.
For a design $\vec{w} \in\mathcal{W}$, the \emph{moment} matrix
$M(\vec{w}):=\sum_{i=1}^s w_i A_iA_i^T$ represents the total
information gained from the design $\vec{w}$.

When the criterion $\Phi$ satisfies certain properties,
problem~\eqref{theBasicProblem} can be interpreted as
selecting the weights $w_i$ that yield the \emph{most accurate
estimation} of $\vecc{\theta}$.
In this paper, we mainly focus on the $D$-optimal problem, where the criterion
$\Phi$ is set to
\begin{equation}
\label{phiD} \Phi_D: M \to(\det M)^{{1}/{m}}.
\end{equation}
In the case of Gaussian measurement error,
this corresponds to the problem of minimizing the volume of the
standard confidence ellipsoid
for the best linear unbiased estimator (BLUE) $\hat{\vecc{\theta}}$
of~$\vecc{\theta}$.

More generally, if the experimenter is interested in the estimation of
the parameter
subsystem $\vecc{\vartheta} = K^T \vecc{\theta}$, where $K$ is an $m
\times k$ matrix ($k\leq m$)
of full column rank [$\operatorname{rank}(K)=k$], a relevant criterion is
$D_K$-optimality, obtained when the $D$-criterion is applied to the
information matrix
$C_K(M)$ for the linear parametric subsystem given by the coefficient
matrix $K$,
defined by (Section~3.2 in~\cite{Puk93})
\[
C_K(M) = \mathop{\mathop{\min\nolimits_{\preceq}}_{L \in\R^{k \times m}}}_{LK=\vec{I}_k}
L M L^T.
\]
Here the minimum is taken with respect to L\"owner ordering, over all
left inverses $L$ of $K$.
This information matrix
is equal to $(K^TM^-K)^{-1}$ if the estimability condition holds
($\operatorname{range} K \subseteq\operatorname{range} M$);
otherwise $C_K(M)$ is a singular matrix, so
\begin{equation}
\label{phiDK} \Phi_{D|K}: M \to \cases{ \displaystyle \bigl(\det
K^T M^- K\bigr)^{-{1}/{k}}, &\quad $\mbox{if }\operatorname{range}  K
\subseteq \operatorname{range} M$;
\vspace*{2pt}\cr
0, &\quad$\mbox{otherwise}$.}
\end{equation}
In the previous formula $M^-$ denotes a generalized inverse of $M$,
that is, a matrix satisfying $MM^- M = M$.
Although $M^-$ is not unique in general, the definition of
$\Phi_{D|K}$ is consistent. Indeed, the matrix $K^T M^- K$ does not
depend on the
choice of the generalized inverse $M^-$ if the columns of $K$ are
included in the range of $M$;
cf. Pukelsheim~\cite{Puk93}.
Note that if $k=1$, that is, if the matrix $K=\vec{c}$ is a nonzero
vector, then the criterion
$\Phi_{D|K}$ is equivalent to the criterion of $\vec{c}$-optimality.



Other optimality criteria, such as
$A$, $A_K$, $G$ and $I$-optimality, are also discussed in the  \hyperref[secothercrits]{Appendix}.


In the standard form of the problem, $\mathcal{W}$ is the probability simplex
\[
\mathcal{W}_\Delta:=\Biggl\{\vec{w} \in\R_+^s: \sum
_{i=1}^s w_i = 1\Biggr\},
\]
and the design $\vec{w}$ is a weight vector indicating the proportions of
trials in the individual design points. This problem, called the \emph
{optimal approximate design problem} in the literature,
is in fact a relaxation of a much more difficult and more fundamental
discrete optimization problem:
the \emph{optimal exact design problem of size $N$},
where
$\mathcal{W}$ takes the form
\[
\mathcal{W}_N:=\Biggl\{\frac{\vec{n}}{N}: \vec{n} \in
\mathbb{N}_0^s, \sum_{i=1}^s
n_i = N\Biggr\}.
\]
Here, the experiment consists of $N$ trials,
and if $\vec{w}\in\mathcal{W}_N$, then $n_i = N w_i$ indicates the
number of trials
in the design point $i$. (In the above definition, $\mathbb{N}_0$
denotes the set of all nonnegative integers, i.e., $0\in\mathbb{N}_0$.)
Note that the constraint $\vec{w} \in\mathcal{W}_\Delta$
is obtained from $\vec{w} \in\mathcal{W}_N$ by relaxing the integer
constraints on $N w_i$.

Many different approaches have been proposed to solve problems of
type \eqref{theBasicProblem}.
However, most methods are specialized and work only if the feasibility
set $\W$
is the probability simplex $\Wdelta$ or the standard discrete simplex
$\WN$.
In the former case (approximate optimal design, $\W=\Wdelta$), the
traditional methods are the Fedorov--Wynn type vertex-direction
algorithms~\cite{Fed72,Wyn70}, and the multiplicative
algorithms~\cite{Tit76,STT78,Yu10a,Yu11},
eventually combined with adaptive changes of the finite grid $\mathcal
{X}$~\cite{HP07,YBT13,PZ14}.
In the latter case (exact optimal design, $\W=\WN$),
the classical methods are heuristics such as
exchange algorithms~\cite{Fed72,Mit74,AD92},
rounding methods~\cite{PR92}
and metaheuristics such as simulated annealing~\cite{Hai87} or genetic
algorithms~\cite{HCMBR03}.
For some small to medium size models,
branch-and-bound methods~\cite{Wel82} have been used to compute
provably optimal solutions.

In many practical situations, however, more complicated constraints
are imposed on the design~\cite{CF95}, and there is a need for more
general algorithms.
For example, assume that the experimental region can be partitioned as
$\mathcal{X}=\mathcal{X}_1 \cup\mathcal{X}_2$,
and that 40\% (resp.,  60\%) of the trials should be chosen in
$\mathcal{X}_1$ (resp.,  $\mathcal{X}_2$);
that is, the constraint $\vec{w} \in\mathcal{W}_\Delta$ is replaced by
\[
\vec{w} \in\mathcal{W}:=\biggl\{\vec{w} \in\R_+^s: \sum
_{i\in\mathcal
{X}_1} w_i = 0.4, \sum
_{i\in\mathcal{X}_2} w_i = 0.6\biggr\}.
\]
This is an example of a \textit{stratified design}~\cite{Har14},
which is a generalization of the well-known
\emph{marginally constrained design}~\cite{CT80}.
Other examples of relevant design domains $\mathcal{W}$ defined by a
set of linear
inequalities are discussed in~\cite{VBW98}.
For example, it is possible to consider a case in which a total budget
is allocated, and the design points
are associated to possibly unequal costs $c_1,\ldots,c_s$.
It is also possible to consider
decreasing costs when trials of specific design points are grouped,
or to avoid designs that are concentrated on a small number of design points.

For some special linear constraints, the approximate $D$-optimal design
problem can be solved by modifications
of the vertex-direction and the multiplicative algorithms (see,
e.g., \cite{CF95,MTL07,Har14}),
but the convergence of these methods is usually slow.
Recently, modern mathematical programming algorithms \cite{VBW98,FL00,HJ08,Sagnol09SOCP,FTH12,Pap12,LP13,Sagnol2013LAA}
have been gaining in popularity.
The idea is to reformulate the optimal design problem
under a canonical form that specialized solvers can handle,
such as \textit{maxdet programs} (MAXDET), \textit{semidefinite programs}
(SDP) or \textit{second-order cone programs} (SOCP).

Reformulating an optimal design problem as an SOCP or an SDP
is useful in many regards. First, it allows one to use modern
software to compute an optimal solution efficiently. Second,
the available interior point methods are
known to return an $\varepsilon$-optimal solution in polynomial time with
respect to the size of the instance
and $\log\frac{1}{\varepsilon}$ because a self-concordant barrier exists
for these problems; cf. \cite{BV04}.
Third, mathematical programming methods are general in the sense that
they are not restricted to the use of
\textit{special linear constraints}.
Nevertheless, the inclusion of general linear constraints within
mathematical programming
characterizations is not completely straightforward.
For instance, we show in Section~\ref{secformer}
that the SOCP formulation of~\cite{Sagnol09SOCP} for the standard
approximate $D$-optimal design problem (over $\mathcal{W}_\Delta$)
does not yield a valid SOCP formulation of the constrained $D$-optimal
design problem when
the constraint $\vec{w} \in\mathcal{W}_\Delta$ is replaced by
$\vec{w} \in\mathcal{W}$.

The main result of this paper is proved in Section~\ref{sectheo} and states
that the determinant criterion is \textit{SOC-representable}.
More precisely, it is possible to express that
$(t,\vec{w})$ belongs to the hypograph of $\vec{w}\to\Phi_D (M(\vec
{w}) )$,
that is, $t^{m}\leq\det M(\vec{w})$,
as a set of second-order cone inequalities.
Consequently, we obtain an alternative SOCP formulation for $D$-optimality,
which remains valid for any weight domain $\mathcal{W}$ that can be
expressed by SOC inequalities; see Section~\ref{secSOCr}.

In the  \hyperref[secothercrits]{Appendix}, we prove that other widely used
criteria, such as $A$, $G$ or $I$-optimality
are also SOC-representable.
We have summarized the 
SOCP formulations of constrained $D$-, $A$- and $G$-optimality in
Table~\ref{tabformulations}.

Before this paper, the state of the art method for solving optimal
design problems with arbitrary linear constraints
was the MAXDET formulation of Vandenberghe, Boyd and Wu \cite{VBW98},
which is in fact
reformulated as an SDP by most interfaces, such as YALMIP \cite{yalmip}
or PICOS \cite{Sagnol12picos},
by using the construction described in \cite{BTN87}.
Having an SOCP instead of an SDP formulation
has two main advantages. The first is purely computational: it is
well known that the computational effort per iteration
required by the interior point methods to solve an SOCP is much less
than that required to solve an SDP; cf. \cite{AG03}.
When the parameter $\vecc{\theta}$ is of large dimension $m$,
or when the number of candidate support points $s$ is large,
the SOCP can improve the computational time by one or two orders of
magnitude (compared to MAXDET),
as was already evidenced in \cite{Sagnol09SOCP} for $D$-optimality over
the probability simplex $\mathcal{W}_\Delta$.

\begin{table}
\caption{SOCP formulation of the $D_K$, $A_K$ and $G$-optimal design problems over a compact weight region $\mathcal{W}
\subseteq\R_+^s$.
In the above, $K$ represents a given
$m \times k$ matrix of full column rank. The particular case $k=1$
(where $\vec{c}=K$ is
a column vector) gives SOCP formulations for the $\vec{c}$-optimal
design problem, and the case $K=\vec{I}_m$ yields
the standard $D$ and $A$-optimality problems.
The variables $Z_i$, $Y_i$
($i\in[s]$) are of size $\ell_i \times k$,
the variables $H_i^j$ ($i\in[s], j\in[s]$) are of size $\ell_j \times
\ell_i$,
$J$ is of size $k\times k$, the weight vector is $\vec{w}\in\mathcal{W}$
and the variables $t_{ij}$ ($i\in[s], j\in[k]$), $u_i^j$ ($i\in[s], j\in
[s]$), $\mu_i$ ($i\in[s]$) and $\rho$ are scalar}
\label{tabformulations}
\begin{tabular*}{\tablewidth}{@{\extracolsep{\fill}}@{\hspace*{43pt}}l@{}}
\hline
$\displaystyle\max_{\vec{w} \in\mathcal{W}} \Phi_{D|K} \bigl(M(\vec{w}) \bigr)
=
\max_{\vec{w},Z_i,t_{ij},J}
\prod_{j=1}^k (J_{j,j})^{{1}/{k}}$\\[15pt]
\hspace*{50pt}\qquad\qquad$\qquad\displaystyle\mbox{s.t.} \qquad  \sum_{i\in[s]} A_i Z_i=KJ$,\\[14pt]
\hspace*{48pt}\qquad\qquad\qquad\phantom{s.t.} \qquad  $ J\mbox{ is lower triangular}$,\\[5pt]
\hspace*{48pt}\qquad\qquad\qquad\phantom{s.t.} \qquad$\displaystyle\Vert Z_i \vec{e}_j \Vert^2 \leq t_{ij} w_i\qquad \bigl( i\in[s], j\in[k]\bigr)$,\\[6pt]
\hspace*{49pt}\qquad\qquad\qquad\phantom{s.t.} \qquad$ \displaystyle\sum_{i=1}^s t_{ij} \leq J_{j,j}\qquad \bigl(j\in[k]\bigr)$,\\[14pt]
\hspace*{49pt}\qquad\qquad\qquad\phantom{s.t.} \qquad$\displaystyle t_{ij} \geq0 \qquad \bigl( i\in[s], j\in[k]\bigr)$,\\[7pt]
\hspace*{48pt}\qquad\qquad\qquad\phantom{s.t.} \qquad$\displaystyle\vec{w} \in\mathcal{W}$,
\\[12pt]
$\displaystyle\max_{\vec{w} \in\mathcal{W}} \Phi_{A|K} \bigl(M(\vec{w}) \bigr)
=
\max_{\vec{w},Y_i,\mu_i}
 \sum_{i\in[s]} \mu_i$\\[12pt]
\hspace*{50pt}\qquad\qquad$\displaystyle\qquad\mbox{s.t.} \qquad \sum_{i\in[s]} A_i Y_i=  \biggl(\sum_{i\in[s]} \mu
_i \biggr) K$, \\[16pt]
\hspace*{31pt}\qquad\qquad\qquad\phantom{s.t.} \qquad$\displaystyle\qquad \Vert Y_i \Vert_F^2 \leq\mu_i w_i \qquad \bigl( i\in[s] \bigr)$,\\[5pt]
\hspace*{32pt}\qquad\qquad\qquad\phantom{s.t.} \qquad$\displaystyle\qquad  \mu_i \geq0\qquad \bigl( i\in[s]\bigr)$, \\[5pt]
\hspace*{31pt}\qquad\qquad\qquad\phantom{s.t.} \qquad$\displaystyle\qquad  \vec{w} \in\mathcal{W}$,
\\[12pt]
$\displaystyle\max_{\vec{w} \in\mathcal{W}} \Phi_{G} \bigl(M(\vec{w}) \bigr) =
\max_{\vec{w},H_i^j,u_i^j,\rho} \rho$\\[14pt]
\hspace*{77pt}$\displaystyle\qquad\mbox{s.t.} \qquad \sum_{j\in[s]} A_j H_i^j=  \biggl(\sum_{j \in[s]}
u_i^j\biggr) A_i\qquad \bigl(i \in[s]\bigr)$, \\[16pt]
\hspace*{88pt}$\displaystyle\qquad\qquad \bigl\Vert H_i^j \bigr\Vert_F^2 \leq w_j u_i^j \qquad \bigl( i\in[s], j\in[s] \bigr)$,\\[4pt]
\qquad\hspace*{89pt}$\displaystyle\qquad  u_i^j \geq0\qquad \bigl( i\in[s], j\in[s]\bigr)$, \\[6pt]
\qquad\hspace*{89pt}$\displaystyle\qquad \rho\leq\sum_{j \in[s]} u_i^j\qquad  \bigl(i \in[s]\bigr)$,\\[14pt]
\qquad\hspace*{89pt}$\displaystyle\qquad  \vec{w} \in\mathcal{W}$. \\
\hline
\end{tabular*}
\end{table}

The second and probably more important benefit of SOCP formulations
(compared to SDP) is that specialized solvers
can handle SOCP problems with integer variables, while
there is currently no reliable solver to handle SDPs with integer variables.
Indeed, much progress has been made recently in the development of
algorithms for second-order cone programming, when some of the
variables are constrained in the integral domain
(MISOCP: mixed integer second-order cone programming).
Thus the SOCP formulation of $D$-optimality
presented in this article,
unlike the existing SOCP and SDP formulations,
allows us to use those specialized
codes to solve exact design problems. Indeed, our formulation is valid
for \textit{any compact weight domain} $\mathcal{W}$, so in
particular it is valid for the set $\mathcal{W}_N$ of exact designs of
size $N$, and more generally
for any polyhedron intersected with a lattice of integer points.
Compared to the raw branch-and-bound method for computing exact designs
proposed by Welch \cite{Wel82}, the MISOCP approach is not only easier to
implement, but also much more efficient.
The reason is that
specialized solvers such as CPLEX \cite{cplex} or MOSEK \cite{mosek}
rely on branch-and-cut algorithms with sophisticated branching heuristics,
and they use \textit{cut inequalities} to separate noninteger solutions.

In Section~\ref{secresults}, we demonstrate the general applicability
of the proposed approach,
incorporating illustrative examples taken from two
application areas of the theory of optimal experimental designs.
The following key aspects of the MISOCP approach will be emphasized:
\begin{longlist}[(3)]
\item[(1)] the ability to handle any system of linear constraints on
the weights;
\item[(2)] the ability to compute exact-optimal designs with a proof of
optimality;
\item[(3)] the ability to rapidly identify a near exact-optimal design
for applications where the computing time
must remain short, while giving a lower bound on its efficiency;
moreover this bound is usually much better
than the standard bound obtained from the approximate optimal design.
\end{longlist}
In particular, our algorithm can compute constrained exact optimal designs,
a feature out of reach of the standard computing methods, although some authors
have proposed heuristics to handle some special cases such as cost
constraints \cite{TV04,WSB10}.
A notable exception is the recent DQ-optimality approach of Harman and
Filov\'{a} \cite{HF13},
which is a heuristic based on integer quadratic programming (IQP) that
can handle
the general case of linearly constrained exact designs.
However, for some specific $D$-optimum design problems, the IQP
approach leads to very inefficient designs; cf. Section~4 in~\cite{HF13}.

In practice, the MISOCP solvers take an input tolerance parameter
$\varepsilon>0$,
and the computation stops when a design $\vec{w}^*$ is found, with a guarantee
that no design $\vec{w}$ with value $\Phi ( M(\vec{w}) )\geq
(1+\varepsilon)\Phi ( M(\vec{w}^*) )$ exists.
In some cases such as $D$-optimal block designs,
there is a positive value of $\varepsilon>0$
for which the returned design is verifiably optimal; see Section~\ref{secresults}.
Otherwise we can set $\varepsilon>0$ to a small constant
(i.e., a tolerance allowing a reasonable computation time),
so the design found with the MISOCP approach
will have an efficiency guarantee of $(1+\varepsilon)^{-1} \geq1-\varepsilon$,
which is usually a much better
efficiency bound than the one based on the comparison with the
approximate optimal design.
In many situations, the solver is further able to terminate with an
\textit{optimality status},
which means that the branch and bound tree has been completely trimmed
and constitutes a proof of optimality.
%
Moreover, it often produces better designs than the standard heuristics
(also in cases when perfect optimality is not guaranteed).


\section{Former SOCP formulation of $D$-optimality}\label{secformer}

A \textit{second-order cone program} (SOCP) is an optimization problem
where a linear function $\vec{f}^T \vec{x}$ must be maximized, among the
vectors $\vec{x}$ belonging to a set $S$-defined by second-order cone
inequalities, that is,
\[
S=\bigl\{\vec{x} \in\R^n: \forall i=1,\ldots,N_c, \Vert
G_i \vec{x} + \vec {h}_i\Vert\leq\vec{c}_i^T
\vec{x} + d_i\bigr\}
\]
for some $G_i,\vec{h}_i,\vec{c}_i,d_i$ of appropriate dimensions.
Optimization problems of this class can be solved efficiently to the
desired precision using interior point techniques; see~\cite{BV04}.

We first recall the result from~\cite{Sagnol09SOCP} about $D$-optimality,
rewritten with the notation of the present article. Note that
$\Vert Z \Vert_F:=\sqrt{\operatorname{trace} ZZ^T}$ denotes the
Frobenius norm of
the matrix $Z$, which also corresponds to the Euclidean norm of the
vectorization of $Z$:
$\Vert Z \Vert_F=\Vert\operatorname{vec}(Z) \Vert$. In the following
formulation,
the restriction to lower triangular matrices is just a compact notation
for the set of
linear constraints that appears in \cite{Sagnol09SOCP}:

\begin{proposition}[(Former SOCP for $D$-optimality \cite{Sagnol09SOCP})]
Let $(Z_1,\ldots, Z_s,\break  L,\vec{w})$ be optimal for the following SOCP:
\begin{eqnarray}
&& \mathop{\mathop{\max_{Z_i \in\R^{\ell_i \times m}}}_{ L \in\R^{m \times
m}}}_{\vec{w} \in\R_+^s}
\Biggl(\prod_{k=1}^m L_{k,k}
\Biggr)^{{1}/{m}}
\nonumber
\\
&&\qquad\mathrm{s.t.}\qquad \sum
_{i=1}^s A_i Z_i = L,
\nonumber
\\
\label{Dopt-oldSOCP} &&\qquad\phantom{s.t.}\qquad L \mbox{ is lower triangular},
\\
&&\qquad\phantom{s.t.}\qquad\Vert Z_i\Vert_F \leq
\sqrt{m} w_i \qquad \forall i \in [s],
\nonumber
\\
&&\qquad\phantom{s.t.}\qquad \vec{w} \in\Wdelta.
\nonumber
\end{eqnarray}

Then $\Phi_D (M(\vec{w}) )= \det^{1/m} M(\vec{w}) =  ( \prod_k L_{k,k}  )^{2/m}$,
and $\vec{w} \in\mathcal{W}_\Delta$ is optimal for the standard
approximate $D$-optimal design problem.
\end{proposition}

If we want to solve a $D$-optimal design problem over another design region
$\mathcal{W}$, it is very tempting to replace the last constraint in
problem~\eqref{Dopt-oldSOCP}
by $\vec{w}\in\mathcal{W}$. However, this approach fails.
Consider, for example, the following experimental\vspace*{1pt} design problem with
three regression vectors in a two-dimensional space:
$A_1=[1,0]^T$, $A_2=[-\frac{1}{2},\frac{\sqrt{3}}{2}]^T,
A_3=[-\frac{1}{2}, -\frac{\sqrt{3}}{2}]^T$.
For reasons of symmetry, it is clear that the approximate $D$-optimal
design (over $\mathcal{W}_\Delta$)
is $w_1=w_2=w_3=\frac{1}{3}$, and this is indeed the vector $\vec{w}$
returned by problem~\eqref{Dopt-oldSOCP}.
Define now\vspace*{1pt} $\mathcal{W}:=\{\vec{w} \in\R_+^3: \sum_{i=1}^3 w_i =1,
w_1 \geq w_2 + 0.25\}$.
The optimal design over $\mathcal{W}$ is $\vec{w}^*=[0.4583,
0.2083, 0.3333]$,
but solving problem~\eqref{Dopt-oldSOCP} with the additional constraint
$w_1 \geq w_2 + 0.25$ yields the
design $\vec{w}=[0.4482,0.1982,0.3536]$, which is suboptimal.


It can be proved that any optimal pair of variables $(\vec{w}^*,L^*)$
for problem~\eqref{Dopt-oldSOCP}
satisfies $M(\vec{w}^*)=(L^*)(L^*)^T$; that is, $L^*$ is a Cholesky
factor of the optimal information matrix.
However, this relation is only true for optimality over the unit
simplex $\mathcal{W}_\Delta$, which is a consequence of
a generalization of Elfving's theorem; cf.~\cite{Sagnol09SOCP}.
In the present article, we give an alternative SOCP formulation of the
$D$-optimal problem,
which remains valid for any compact
weight domain $\mathcal{W}$.
The main idea of our new formulation is that the Cholesky factorization
of a matrix $HH^T$ can be computed by solving an SOCP that mimics the
Gram--Schmidt orthogonalization process of the rows of $H$.
Moreover, our new SOCP handles the more general case of $D_K$-optimality.
To derive our result, we use the notion of \emph{SOC-representability}, which
we next present.

\section{SOC-representability} \label{secSOCr}

In this section, we briefly review some basic notions about
second-order cone representability. The following definition
was introduced by Ben-Tal and Nemirovski~\cite{BTN87}:

\begin{definition}[(SOC-representability of a set)]
A convex set $S \subseteq\R^n$ is said to be \textit{second-order cone
representable},
abbreviated \textit{SOC-representable}, if $S$ is the projection
of a set in a higher-dimensional space that can be described by
a set of second-order cone inequalities. More precisely, $S$ is
SOC-representable if and only if there exist
$G_i \in\R^{n_i \times(n+m)}, \vec{h}_i \in\R^{n_i},
\vec{c}_i \in\R^{n+m}, d_i \in\R$ ($i=1,\ldots,N_c$), such that
\[
\vec{x} \in S \quad\Longleftrightarrow\quad\exists\vec{y} \in\R^m:
\forall i=1,\ldots,N_c,\qquad \biggl\llVert G_i \lleft[
\matrix{ \vec{x}
\cr
\vec{y}} \rright] + \vec{h}_i \biggr\rrVert \leq
\vec{c}_i^T \lleft[ \matrix{ \vec{x}
\cr
\vec{y}}
\rright] +d_i.
\]
\end{definition}

An important example of an SOC-representable set is the
following:

\begin{lemma}[(Rotated second-order cone inequalities)]
\label{lemmarotated}
The set
\[
S=\bigl\{(\vec{x} , t, u)\in\R^n \times\R\times\R: \Vert\vec{x}
\Vert^2 \leq t u,  t\geq0, u\geq0\bigr\} \subseteq\R^{n+2}
\]
is SOC-representable. In fact, it is easy to see that
\[
S=\biggl\{(\vec{x},t,u)\in\R^n\times\R\times\R: \biggl\llVert %
\matrix{ 2\vec{x}
\cr
t-u } %
\biggr\rrVert \leq t+u\biggr\}.
\]
\end{lemma}

The notion of SOC-representability is also defined for functions:

\begin{definition}[(SOC-representability of a function)]
A convex (resp.,  concave) function $f: S \subseteq\R^n \mapsto\R
$ is said to be SOC-representable if and only if
the epigraph of $f$, $\{(t,\vec{x}): f(\vec{x}) \leq t\}$
[resp., the hypograph $\{(t,\vec{x}): t \leq f(\vec{x})\}$],
is SOC-representable.
\end{definition}

It follows immediately from these two definitions that
the problem of maximizing a concave SOC-representable function
(or minimizing a convex one)
over an SOC-representable set can be cast as an SOCP.
It is also easy to verify that sets defined by linear equalities (i.e.,
polyhedrons) are SOC-representable,
that intersections of SOC-representable sets are SOC-\break representable
and that the (pointwise) minimum of concave SOC-\break representable functions
is still concave and SOC-representable.

We next give another example
which is of major importance for this article:
the geometric mean of $n$ nonnegative variables is SOC-representable.

\begin{lemma}[(SOC-representability of a geometric mean~\cite{BTN87})]
\label{lemmaproduct}
Let $n\geq1$ be an integer.
The function $f$ mapping $\vec{x}\in\R_+^n$ to $\prod_{i=1}^n x_i^{1/n}$
is SOC-representable.
\end{lemma}

For construction of the SOC representation of $f$, see~\cite{LVBL98}
or~\cite{AG03}. In what follows, we show the case $n=5$. For all
$t\in\R_+$,
$\vec{x} \in\R_+^5$, we have
\begin{eqnarray*}
t^5 \leq x_1 x_2 x_3
x_4 x_5\quad & \Longleftrightarrow& \quad
t^8 \leq x_1 x_2 x_3
x_4 x_5 t^3
\\
& \Longleftrightarrow& \quad\exists\vec{u} \in\R_+^5: \cases{
u_1^2 \leq x_1 x_2,
& \quad $u_4^2 \leq u_1 u_2$,
\vspace*{3pt}
\cr
u_2^2 \leq x_3
x_4, &\quad $u_5^2 \leq u_3 t$,
\vspace*{3pt}
\cr
u_3^2 \leq x_5 t, & \quad
$t^2 \leq u_4 u_5$,}
\end{eqnarray*}
and each of these inequalities can be transformed to
a standard second-order cone inequality by Lemma~\ref{lemmarotated}.

\section{SOC-representability of the $D$-criterion} \label{sectheo}

The key to SOC representation of the $D$-criterion is a Cholesky
decomposition of the moment matrix,
as given by the following lemma. Note that the lemma is general in the
sense that it does not require the estimability
conditions to be satisfied.

\begin{lemma}\label{CholeskyInf}
Let $H$ be an $m \times n$ matrix ($m \leq n$), and let $K$ be an $m
\times k$ matrix ($k \leq m$) of full column rank.
If $k=m$, let $U=K$, and if $k<m$, let $U$ be a
nonsingular matrix of the form $[V,K]$, where $V\in\R^{m \times(m-k)}$.
Then there exists a QR-decomposition\vspace*{1pt} of $H^TU^{-T}=\tilde{Q}\tilde{R}$
where $\tilde{Q}$ is an orthogonal $n \times n$ matrix and $\tilde{R}$
is an upper triangular $n \times m$ matrix,
satisfying $\tilde{R}_{ii} \geq0$ for all $i \in[m]$ and
\begin{equation}
\label{zerorows} \tilde{R}_{ii}=0\qquad \mbox{implies }
\tilde{R}_{i1}=\cdots=\tilde {R}_{im}=0 \qquad \mbox{for all }
i \in[m].
\end{equation}
Let $L_*^T$ be the $k \times k$ upper triangular sub-block of $\tilde
{R}$ with elements
$(L^T_*)_{ij}=\tilde{R}_{m-k+i,m-k+j}$ for all $i,j \in[k]$. Then
$C_K(HH^T)=L_*L_*^T$; that is, $L_*L_*^T$ is a
Cholesky factorization of the information matrix for the linear
parametric system given by the coefficient matrix $K$, corresponding to
the moment
matrix $HH^T$.
\end{lemma}

\begin{pf}
It is simple to show that a QR decomposition satisfying \eqref
{zerorows} can be obtained from any
QR-decomposition $H^TU^{-T}=\bar{Q}\bar{R}$, using an appropriate
sequence of Givens rotations\vspace*{1pt} and row permutations applied on $\bar{R}$.

Consider the decomposition $H^TU^{-T}=\tilde{Q}\tilde{R}$ satisfying
\eqref{zerorows}.
Assume that $k<m<n$. Partition the orthogonal matrix $\tilde{Q}$ and
the upper
triangular matrix $\tilde{R}$ as follows:
\begin{equation}
\label{blocks} 
\tilde{Q}= %
\begin{array} {c@{\quad}c@{
\quad}c@{\quad}l} \mbox{\scriptsize$\displaystyle\mathop{\longleftrightarrow}^{m-k}$}
\!\! & \mbox{\scriptsize$\displaystyle\mathop{\longleftrightarrow}^{k}$}
& \!\!\mbox{\scriptsize$\displaystyle\mathop{\longleftrightarrow}^{n-m}$}
&
\cr
[Q_1 & Q_* & Q_2] & \mbox{
\scriptsize$\updownarrow n$}, \end{array} %
\qquad 
\tilde{R}= \begin{tabular} {c} \hspace*{-39pt}$\matrix{\mbox{\scriptsize$
\displaystyle\mathop{\longleftrightarrow}^{m-k}$} & \!\mbox{\scriptsize$
\displaystyle\mathop{\longleftrightarrow}^{k}$} }$ \vspace*{2pt}
\\
$\lleft[\matrix{ L_1^T & B
\cr
0 &
L_*^T
\cr
0 & 0 } \rright] \hspace*{6pt}\matrix{\mbox{
\scriptsize$\hspace*{-3pt}\updownarrow m-k$}
\cr
\hspace*{-16pt}\mbox{\scriptsize$\updownarrow
k$}
\cr
\hspace*{0pt}\mbox{\scriptsize$\updownarrow n-m$}}$ \end{tabular}
\end{equation}
where the block sizes are indicated on the border of the matrices.
Let $U^{-1}=[Z^T, X^T]^T$, where $X$ is a $k \times m$ matrix. Note
that $[Z^T,X^T]^TK=U^{-1}K=[0,\vec{I}_k]^T$, which implies $XK=\vec
{I}_k$; that is,
$X$ is a left inverse of $K$.
Define $Y=\vec{I}_m-KX$. By a direct calculation, we obtain
$XH=B^TQ_1^T+L_*Q_*^T$ and $YH=H-KXH=[V,K]\tilde{R}^T\tilde
{Q}^T-K(B^TQ_1^T+L_*Q_*^T)=VL_1Q_1^T$.
Therefore, using the orthogonality of $\tilde{Q}$, that is,
$Q_1^TQ_1=\vec{I}_{m-k}$, $Q_*^TQ_*=\vec{I}_k$, $Q_1^T Q_* = 0$
and a representation of $C_K$ given by \cite{Puk93}, Section~3.2,
we have
\begin{eqnarray}
C_K\bigl(HH^T\bigr)&=&XHH^TX^T-XHH^TY^T
\bigl(YHH^TY^T\bigr)^-YHH^TX^T
\nonumber
\\[-9pt]
\label{Pukinf}
\\[-9pt]
\nonumber
&=&B^TB+L_*L_*^T - B^T
\underbrace{L_1^T V^T \bigl(VL_1L_1^T
V^T\bigr)^- VL_1}_{P} B, 
\end{eqnarray}
where $P$ is the orthogonal projector on $\range(L_1^TV^T)$.
Note that \eqref{zerorows} implies $\range(B) \subseteq\range
(L_1^T)$, and $\operatorname{rank}(V)=m-k$ gives $\range(L_1^T)=\range
(L_1^TV^T)$.
That is, $PB=B$, and from \eqref{Pukinf} we obtain the required result
$C_K(HH^T)=L_*L_*^T$.

If $k=m$ or $n=m$, the lemma can be proved in a completely analogous
way, treating the matrices $Q_1,L_1,B$ (if and only if $k=m$) and $Q_2$
(if and only if $m=n$) as \textit{empty}.\vspace*{-3pt}
\end{pf}

The next theorem shows that the blocks $Q_*$ and $L_*$ from
decomposition~\eqref{blocks} can be computed
by solving an optimization problem over an SOC-representable set.\vspace*{-3pt}

\begin{theorem} \label{theoKdet}
Let $H$ be an $m \times n$ matrix ($m \leq n$),
let $K$ be an $m \times k$ matrix ($k \leq m$) of full column rank and
let $L_*$ be optimal for the following problem:
\begin{eqnarray}
&&\mathop{\max_{Q \in\mathbb{R}^{n \times k}}}_{L \in\mathbb{R}^{k \times k}} \det L
\nonumber\\[-2pt]
&&\qquad\mathrm{s.t.} \qquad L \mbox{ is lower triangular},
\nonumber
\\[-9pt]
\label{probKdet}
\\[-9pt]
\nonumber
&&\hspace*{-2pt}\qquad\phantom{s.t.}\qquad HQ = KL,
\\[-2pt]
&&\hspace*{-2pt}\qquad\phantom{s.t.}\qquad\Vert Q\vec{e}_j \Vert\leq 1 \qquad \bigl(j\in[k]\bigr).
\nonumber
\end{eqnarray}

Then\vadjust{\goodbreak} $\Phi_{D|K}(HH^T)=(\det(L_*))^{2/k}$.
\end{theorem}
\begin{pf}
Consider the QR decomposition $H^TU^{-T}=\tilde{Q}\tilde{R}$ from the
statement of Lemma~\ref{CholeskyInf},
and the block partition~\eqref{blocks}. We will show that the blocks
$Q_*$ and $L_*$ form an optimal solution to the problem from the theorem.

First,\vspace*{1pt} $L_*$ is clearly lower triangular, and using direct block
multiplication together with $Q_*^TQ_*=\vec{I}_k$, we can verify that
$Q_*^TH^T=L_*^TK^T$, that is,
$HQ_*=KL_*$. Second, $Q_*$ has columns of unit length, which implies $\|
Q_*\vec{e}_j\| = 1$ for all $j\in[k]$.
Therefore, $Q_*,L_*$ are feasible. From Lemma~\ref{CholeskyInf}, we
know that $C_K(HH^T)=L_*L_*^T$, that is,
$(\det(L_*))^{2/k}=\Phi_{D|K}(HH^T)$. To\vspace*{1pt} complete the proof of the
theorem, we only need to show that any
feasible $L$ satisfies $(\det(L))^{2/k} \leq\Phi_{D|K}(HH^T)$.

Let $Q,L$ be a feasible pair of matrices. As in the proof of Lemma~\ref
{CholeskyInf}, let $U=[V,K]$ be an invertible matrix, and let
$U^{-1}=[Z^T,X^T]^T$,
where $X$ is a $k \times m$ matrix. Obviously, $U^{-1}H=[C^T,D^T]^T$,
where $C=ZH$ and $D=XH$, and $[C^T,D^T]^TQ=U^{-1}HQ=U^{-1}KL=[0,\vec
{I}_k]^TL=[0,L^T]^T$,
which implies $CQ=0$ and $DQ=L$. Define the projector $P=\vec
{I}_n-C^T(CC^T)^-C$, that is, $P^2=P$, and then observe that $CQ=0$
entails $PQ=Q$.
From the previous equalities and the Cauchy--Schwarz inequality for
determinants [e.g., \cite{Seb08}, formula~12.5(c)], we have
\begin{equation}
\label{Eq1}\hspace*{9pt} \det\bigl(LL^T\bigr)=\bigl(\det(DQ)
\bigr)^2=\bigl(\det(DPQ)\bigr)^2 \leq\det
\bigl(DPD^T\bigr)\det\bigl(Q^TQ\bigr).
\end{equation}
The Hadamard determinant inequality (e.g., \cite{Seb08}, formula
12.27) and the feasibility of $Q$ give
\begin{equation}
\label{Eq2} \det\bigl(Q^TQ\bigr) \leq\prod
_{i=1}^k \bigl(Q^TQ\bigr)_{ii}
= \prod_{i=1}^k \|Q\vec {e}_i
\|^2 \leq1.
\end{equation}
Combining\vspace*{1pt} \eqref{Eq1} and \eqref{Eq2}, we obtain $\det(LL^T) \leq\det
(DPD^T)$, and the proof will be complete,\vspace*{1pt} once we prove $DPD^T=C_K(HH^T)$.

Note that\vspace*{1pt} $\vec{I}_m=UU^{-1}=[V,K][Z^T,X^T]^T=VZ+KX$, that is, $Y:= \vec
{I}_m-KX = VZ$.
Moreover, $\operatorname{rank}(V)=m-k$ implies $\operatorname
{range}(H^TY^T)=\operatorname{range}(H^T\times\break Z^T V^T)=\operatorname{range}(H^TZ^T)$;
that is,\vspace*{1pt} the orthogonal projectors $H^TY^T(YHH^T\times\break Y^T)^-HY$ and
$H^TZ^T(ZHH^TZ^T)^-HZ$ coincide. Consequently, using \cite{Puk93},
Section~3.2, we have
\begin{eqnarray*}
C_K\bigl(HH^T\bigr)&=&XHH^TX^T-XHH^TY^T
\bigl(YHH^TY^T\bigr)^-YHH^TX^T
\\
&=&XHH^TX^T-XHH^TZ^T
\bigl(ZHH^TZ^T\bigr)^-ZHH^TZ^T
\\
&=&DD^T-DC^T\bigl(CC^T\bigr)^-CD^T=DPD^T.
\end{eqnarray*}
\upqed\end{pf}

We next apply Theorem~\ref{theoKdet} to the matrix
$H=[\sqrt{w_1} A_1,\ldots,\sqrt{w_s} A_s]$.
This will allow us to express
$\Phi_{D|K} (M(\vec{w}) )$
as the optimal value of an SOCP.
Moreover, we make a change of variables which transforms the
optimization problem into an SOCP where $\vec{w}$
may play the role of a variable.

\begin{theorem} \label{theodetsocr}
Let $K$ be an $m \times k$ matrix $(k\leq m)$
of full column rank.
For all nonnegative weight vectors $\vec{w} \in\R_+^s$,
denote by $\mathit{OPT}(\vec{w})$ the optimal value
of the following optimization problem, where
the optimization variables are $t_{ij} \in\R_+$ ($\forall i\in
[s],\forall j\in[k]$),
$Z_i \in\R^{\ell_i \times k}$ $(\forall i\in[s])$
and $J\in\R^{k\times k}$:
%
\renewcommand{\theequation}{4.7\alph{equation}}
\setcounter{equation}{0}
\begin{eqnarray}
\label{socpDopt}\label{socpdetobj} && \max_{Z_i,t_{ij},J} \Biggl(\prod
_{j=1}^k J_{j,j}
\Biggr)^{1/k}
\\
\label{socpdetazl} &&\qquad\mbox{s.t.}\qquad \sum_{i=1}^s
A_i Z_i=KJ,
\\
\label{socpdetlowtri} &&\qquad\phantom{s.t.}\qquad J \mbox{ is lower triangular},
\\
\label{socpdetnorm} &&\qquad\phantom{s.t.}\qquad \Vert Z_i
\vec{e}_j \Vert^2 \leq t_{ij} w_i
\qquad \bigl( i\in[s], j\in[k]\bigr),
\\
\label{socpdetsum1} &&\qquad\phantom{s.t.}\qquad \sum_{i=1}^s
t_{ij} \leq J_{j,j}\qquad \bigl(j\in[k]\bigr).
\end{eqnarray}
%
Then we have
\[
\mathit{OPT}(\vec{w}) = \Phi_{D|K} \bigl(M(\vec{w}) \bigr).
\]
%
\end{theorem}
\begin{pf}
Let $\vec{w}\in\R_+^s$,
and define $H:=[\sqrt{w_1} A_1,\ldots,\sqrt{w_s} A_s]$.
We are going to show that every feasible solution to problem
\eqref{socpDopt}--\eqref{socpdetsum1}
yields a feasible solution for problem~\eqref{probKdet}
in which\vspace*{1pt} $J_{j,j} = L_{j,j}^2$ for all $j\in[k]$,
and vice versa.
Hence the optimal value of problem~\eqref{socpDopt}--\eqref{socpdetsum1} is
\begin{eqnarray*}
\mathit{OPT}(\vec{w}) &=& (\det J)^{1/k}= (\det L)^{2/k}=
\Phi_{D|K}\bigl(HH^T\bigr)= \Phi _{D|K} \bigl( M(
\vec{w}) \bigr), 
\end{eqnarray*}
from which the conclusion follows.

Consider a feasible solution $(Z_i,t_{ij},J)$ to problem~\eqref{socpDopt}--\eqref{socpdetsum1}.
We denote by $\vec{z}_{ij}$
the $j$th column of $Z_i$: $\vec{z}_{ij}:=Z_i\vec{e}_j$.
We now make the following
change of variables: denote by $Q_i$ the matrix whose
$j$th column is $\vec{q}_{ij}$,
where
\[
\vec{q}_{ij}=\cases{ %
\displaystyle\frac{\vec{z}_{ij}}{\sqrt{w_i}\sqrt{J_{j,j}}}, &\quad$\mbox{if }
w_i>0 \mbox{ and } J_{j,j}>0$; \vspace*{4pt}
\cr
\vec{0}, &$
\quad\mbox{otherwise}$,}
\]
and define $Q$ as the vertical concatenation of the $Q_i$: $Q =
[Q_1^T,\ldots,Q_s^T]^T$.
Let $j\in[k]$. If $J_{j,j}=0$, then $\vec{q}_{ij}=\vec{0}$ for all $i$, so
$\Vert Q\vec{e}_j \Vert^2= \sum_i \Vert\vec{q}_{ij} \Vert^2 = 0 \leq
1$. Otherwise ($J_{j,j}>0$),
constraint~\eqref{socpdetnorm} together with the nonnegativity of
$t_{ij}$ implies $\Vert\vec{q}_{ij} \Vert^2 \leq\frac{t_{ij}}{J_{j,j}}$,
and by constraint~\eqref{socpdetsum1}, we must have
\[
\Vert Q\vec{e}_j \Vert^2 = \sum
_i \Vert\vec{q}_{ij} \Vert^2 \leq\sum
_i \frac{t_{ij}}{J_{j,j}} \leq1.
\]

Observe that constraints~\eqref{socpdetnorm} and~\eqref{socpdetsum1}
also imply that $\vec{z}_{ij}=\vec{0}$ whenever
$w_i=0$ or $J_{j,j}=0$, so that for all $i\in[s]$, $j\in[k]$, we can
write $\vec{z}_{ij} = \sqrt{w_i} \sqrt{J_{j,j}} \vec{q}_{ij}$.
Now, we define the matrix $L$ column-wise as follows:
\[
\forall j\in[k],\qquad L\vec{e}_j:= \cases{ %
\displaystyle\frac{J \vec{e}_j}{\sqrt{J_{j,j}}}, &\quad$\mbox{if } J_{j,j}>0$; \vspace*{4pt}
\cr
\vec{0},
&$\quad\mbox{otherwise}$.}
\]
Note that $L$ is lower triangular [because so is $J$; see~\eqref
{socpdetlowtri}].
We can now prove that $HQ=KL$, 
which we do column-wise.\vspace*{1pt}
If $J_{j,j}=0$, then we know that $Q\vec{e}_j=\vec{0}$, so the $j$th
columns of $HQ$ and $KL$ are zero.
If $J_{j,j}>0$, then using~\eqref{socpdetazl} we have
\[
KL\vec{e}_j = \frac{KJ \vec{e}_j}{\sqrt{J_{j,j}}} = \frac{\sum_i A_i \vec{z}_{ij}}{\sqrt{J_{j,j}}} = \sum
_i \sqrt{w_i} A_i
\vec{q}_{ij} =HQ\vec{e}_j.
\]
Hence the proposed change of variables transforms
a feasible solution $(Z,t_{ij},J)$ to
problem~\eqref{socpDopt}--\eqref{socpdetsum1} into a feasible pair $(Q,L)$
for problem~\eqref{probKdet}, with the property
$J_{j,j} = L_{j,j}^2$ for all\vspace*{1pt} $j\in[k]$.

Conversely, let $(Q,L)$ be feasible for problem~\eqref{probKdet},
where $H$ has been set to $[\sqrt{w_1} A_1,\ldots,\sqrt{w_s} A_s]$.
For $i\in[s]$, define $Z_i$ as the matrix of size $\ell_i \times k$
whose $j$th column
is $\vec{z}_{ij} = \sqrt{w_i} L_{j,j} \vec{q}_{ij}$,
and $J$\vspace*{1pt} as the lower triangular matrix whose $j$th column is $J\vec
{e}_j=L_{j,j} L\vec{e}_j$.
We have $\sum_i A_i Z_i=KJ$, which can be verified column-wise as follows:
\[
KJ\vec{e}_j = L_{j,j} KL\vec{e}_j =
L_{j,j} HQ\vec{e}_j = L_{j,j} \sum
_i \sqrt{w_i} A_i
\vec{q}_{ij} = \sum_i A_i
\vec{z}_{ij} = \sum_i A_i
Z_i \vec{e}_j.
\]
Define further $t_{ij} = L_{j,j}^2 \Vert\vec{q}_{ij} \Vert^2$,
so that constraints~\eqref{socpdetnorm} and~\eqref{socpdetsum1} hold.
This shows that $(Z_i,t_{ij},J)$ is feasible, with $J_{j,j} =
L_{j,j}^2$ for all $j\in[k]$, and the proof is complete.
\end{pf}

\begin{corollary}[(SOC-representability of $\Phi_{D|K}$)]\label
{cororepresentable}
For any $m \times k$ matrix $K$ of rank $k$, the function
\mbox{$\vec{w} \to\Phi_{D|K} (M(\vec{w}) )$} is SOC-representable.
\end{corollary}
\begin{pf}
Problem~\eqref{socpDopt}--\eqref{socpdetsum1} can be reformulated as an SOCP,
because by Lemmas~\ref{lemmaproduct}
and~\ref{lemmarotated}
the geometric mean in~\eqref{socpdetobj}
and inequalities of the form $\Vert Z_i \vec{e}_j \Vert^2 \leq t_{ij} w_i$
are SOC-representable.
Hence the optimal value of~\eqref{socpDopt}--\eqref{socpdetsum1},
$\vec{w} \to \mathit{OPT}(\vec{w})$, is SOC-representable, and we know
from Theorem~\ref{theodetsocr} that $\mathit{OPT}(\vec{w}) = \Phi_{D|K}
(M(\vec{w}) )$.
\end{pf}

\begin{corollary}[{[(MI)SOCP formulation of the $D$-optimal design
problem]}]\label{coroMISOCP}
If the set $\mathcal{W}$ is SOC-representable (in particular, if
$\mathcal{W}$ is defined by a set of linear inequalities),
then the constrained $D_K$-optimal design problem~\eqref
{theBasicProblem} can be cast as an SOCP.
If $\mathcal{W}$ is the intersection of an SOC-representable
set with the integer lattice $\mathbb{Z}^s$, then the \textit{exact}
$D_K$-optimal design problem over $\mathcal{W}$ can be cast as an MISOCP.
\end{corollary}


For $K=\vec{I}_m$, Corollaries~\ref{cororepresentable} and~\ref
{coroMISOCP} cover the case of the standard $D$-optimality.
The (MI)SOCP formulation of problem~\eqref{theBasicProblem} for
$D_K$-optimality ($\Phi= \Phi_{D|K}$)
is summarized in Table~\ref{tabformulations}, together with
formulations for the other criteria presented in the \hyperref[secothercrits]{Appendix}.
Finally, we note that the SOCP formulation of the optimal design
problem with constraints on the weights
has consequences in terms of complexity, which we next present.


\subsection*{Complexity of computing constrained approximate
$D_K$-optimal designs}
Recall that $s$ denotes the number of candidate support points, and
$k\leq m$ denotes the number of features
that we wish to estimate. (The full rank coefficient matrix $K$ is in
$\mathbb{R}^{m \times k}$.)
Assume for simplicity that $\ell_i=\ell$ for all $i \in[s]$,
that the set of design weights $\mathcal{W}$ is defined by a set of $n$
inequalities
and that $k$ is a power of $2$, so that the geometric mean can be
represented by $k$ inequalities
and $k$ auxiliary variables;
cf. Lemma~\ref{lemmaproduct} or \cite{Sagnol2013LAA} for more details.
Then the SOCP formulation for $D_K$-optimality of Table~\ref{tabformulations} contains:
\begin{itemize}
\item $s+s\ell k+ s k + \frac{1}{2} k(k+1) + k$ variables,
\item $mk + k + n$ linear (in)equalities,
\item $k$ SOC inequalities of size $2$ and $ks$ SOC inequalities of
size $\ell+1$.
\end{itemize}
The number of iterations required by the interior point methods (IPM)
to compute an $\varepsilon$-approximate solution
depends only on the number $q$ of second-order cones. Indeed it is
shown in~\cite{BTN87}
that the IPM finds an $\varepsilon$-approximate\vspace*{1.5pt} solution after at most
$\sqrt{q}O(\log\frac{1}{\varepsilon})$ iterations, which is
$\sqrt{k(s+1)}O(\log\frac{1}{\varepsilon})$ iterations in our setting.
However, it is well known that this bound is overconservative, and in
practice the IPM always returns
an excellent solution after 10 to 40 iterations, almost independently
of the problem size.
In other words, the critical point is the algorithmic complexity of one
iteration.
Again, a result of~\cite{BTN87} (Section~4.6.2)
allows us to bound the number of algorithmic operations for one
iteration in
$O (ks\ell ((ks\ell)^2 + (mk+n)^2 )  )$,
which is $O((ks\ell)^3)$ if $m$ and $n$ are not too large. But it is
well known that this bound is very conservative, too.
In fact, the bottleneck of one iteration is the resolution of a linear
system of the form
$B\vecc{\delta}=\vecc{\beta}$, where $B$ is a $O(ks\ell) \times O(ks\ell
)$ symmetric positive semidefinite matrix.
In practice, for SOCPs the matrix $B$ has a ``\textit{diagonal} $+$ \textit{sparse
low rank}'' structure,
which allows for an efficient computation of the Newton direction
$\vecc{\delta}$ \cite{AG03}.

\section{Examples} \label{secresults}

In this section, we will present numerical results for several examples
taken from various application areas of the theory of optimal designs.
With these examples, we aim to demonstrate the general applicability of
the (MI)SOCP technique
for the computation of exact or approximate $D$-optimal designs.

Our computations were conducted on a PC with a 4-core processor at 3~GHz.
We used MOSEK~\cite{mosek} to solve the approximate optimal design problems
and CPLEX~\cite{cplex} for the exact optimal design problems (with
integer constraints). The solvers were interfaced
through the Python package PICOS~\cite{Sagnol12picos}, which allows
users to
pass (MI)SOCP models 
to different solvers in a simple fashion. We refer the reader to the
example section of the
PICOS documentation for a practical implementation of the (MI)SOCP
approach for optimal design problems.

It is common to compare several designs against each other by using the
metric of $D$-efficiency, which is defined as
\[
\mathrm{eff}_D(\vec{w}) = \frac{\Phi_D (M(\vec{w}) )}{\Phi_D (M(\vec{w}^*) )} = \biggl(
\frac{\det M(\vec{w})}{\det M(\vec{w}^*)} \biggr)^{1/m},
\]
where $\vec{w}^*$ is a reference design, such that $M(\vec{w}^*)$ is
nonsingular. Unless stated otherwise, we always give $D$-efficiencies
relative to the optimal design; that is, $\vec{w}^*$~is a solution to
problem~\eqref{theBasicProblem}.

\subsection*{Block designs with blocks of size two}
An important category of models studied in the experimental design
literature is the class of \textit{block designs}. Here the effect of
$t$ treatments should be compared, but their effects can only be measured
inside a number $b$ of \textit{blocks}, each inducing a block effect on
the measurements.
The optimal design problem entails choosing which treatments should be
tested together in each block.
We refer the reader to Bailey and Cameron~\cite{BC09} for a
comprehensive review on the combinatorics of
block designs.

In the case where the blocks are of size two, that is, the treatments
can be tested pairwise against each other,
a design can be represented by a vector $\vec{w} =
[w_{1,2},w_{1,3},\ldots,w_{1,t},\ldots,w_{t-1,t}]$
of size $s={t \choose 2}$. For $i<j$, $w_{i,j}$ indicates the number of
blocks where treatments $i$ and $j$
are tested simultaneously. The observation matrix associated with the
block $(i,j)$ can be chosen as the column vector
of dimension $m=(t-1)$,
\begin{equation}
\label{obssparse} A_{i,j} = P (\vec{e}_i -
\vec{e}_j),
\end{equation}
where $\vec{e}_i$ denotes the $i$th unit vector in the canonical basis
of $\R^t$ and $P$ is the
matrix 
that transforms a $t$-dimensional vector $\vec{v}$ to the vector
obtained by keeping the first $(t-1)$ coordinates
of $\vec{v}$.

The problem of $D$-optimality has a nice graph theoretic interpretation:
let $\vec{w}\in\mathbb{N}_0^s$ be a feasible block design, and denote
by $G$
the graph with $t$ vertices and an edge of multiplicity $w_{i,j}$ for
every pair of
nodes $(i,j)$. (If $w_{i,j} = 0$, then there is no edge from $i$ to $j$.)
This graph is called the \textit{concurrence graph} of the design.
We have $M(\vec{w})=PL(\vec{w})P^T$,
where\vspace*{1pt} $L(\vec{w}):=\sum_{i,j} w_{i,j} (\vec{e}_i -\vec{e}_j) (\vec{e}_i
-\vec{e}_j)^T \in\R^{t\times t}$ is the Laplacian of $G$.
In other words, $M(\vec{w})$
is the submatrix of the Laplacian of $G$
obtained by removing its last row and last column.
So by Kirchhoff's theorem the determinant of
$M(\vec{w})$ is the number of spanning trees of $G$. In other words, the
exact $D$-optimal designs of size $N$ correspond to the graphs with $t$
nodes and $N$ edges
that have a maximum number of spanning trees.

\begin{remark} \label{altparam}
There is an alternative parametrization of block designs with blocks of
size two; see~\cite{HF13}. Define the observation matrices by
\begin{equation}
\label{obsproj} A'_{i,j} = U^T (
\vec{e}_i - \vec{e}_j),
\end{equation}
where the columns of $U\in\R^{t \times(t-1)}$ form an orthonormal
basis of $\Ker\vec{1}$ ($\vec{1}$ is the vector with all components
equal to $1$);
that is, the $t\times t$-matrix $[U,\frac{1}{\sqrt{t}}\mathbf{1}]$ is
orthogonal.
It can be seen that the $t-1$ eigenvalues of $M'(\vec{w}) = \sum_{i,j} w_{i,j} A_{i,j}' A_{i,j}'^T = U^TL(\vec{w}) U$ coincide with the
$t-1$ largest
eigenvalues of $L(\vec{w})$, and the smallest eigenvalue of $L(\vec
{w})$ is $0$. So the set of $D$-optimal designs
for observation models~\eqref{obssparse} and~\eqref{obsproj}
coincide. In our experiments,
we have used the former model~\eqref{obssparse} because it involves
sparse information matrices and yields
more efficient computations. However, note that for some other criteria
depending on the eigenvalues of the information matrix,
the model given by~\eqref{obsproj} should be used.
\end{remark}

To illustrate the new capability of the MISOCP approach, we computed designs
of $N=15$ blocks on $t=10$ treatments by imposing different types of constraints
on the replication numbers
(i.e., the
numbers of times that each treatment is tested).
Such constraints can be easily expressed by linear (in)equalities. For example,
a design $\vec{w}$ has treatment~$j$ replicated $r_j$ times if and only if
\[
\sum_{i=1}^{j-1} w_{i,j} + \sum
_{i=j+1}^{t} w_{j,i} =
r_j.
\]
The concurrence graphs of these constrained optimal designs are displayed
in Figure~\ref{figconcgraphs}.
Note that these constrained exact optimal designs cannot be computed by
any of the standard methods.

\begin{figure}
\begin{tabular}{@{}c@{}}

\includegraphics{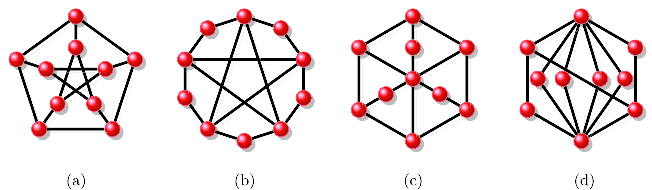}

\\[6pt]
\footnotesize{\begin{tabular*}{\tablewidth}{@{\extracolsep{\fill}}@{}lcccc@{}}
\hline
\textbf{Design}&\textbf{(a)}&\textbf{(b)}&\textbf{(c)}&\textbf{(d)}\\
\hline
CPU (s) & \phantom{00}9.07\phantom{\%} & \phantom{0}4.9\phantom{0\%} & 13.8\phantom{0\%} & \phantom{00}5.7\phantom{0\%}\\
Lower bound on $\mathrm{eff}_D$ (initial) & \phantom{0}90.15\% & 92.56\% &
91.36\% & \phantom{0}91.27\%\\
Lower bound on $\mathrm{eff}_D$ (10 min) & 100.0\%\phantom{0} & 96.56\% &
98.04\% & 100.0\%\phantom{0}\\
\hline
\end{tabular*}}
\end{tabular}
\caption{Concurrence graphs of the $D$-optimal designs of $N=15$ blocks
on $t=10$ treatments,
among the class of $2$-block designs that \textup{(a)} are equireplicate;
\textup{(b)} have half of the treatments replicated $2$ times, and the other
half replicated $4$ times;
\textup{(c)} have one treatment replicated at least $6$ times;
\textup{(d)} have two treatments replicated at least $6$ times.
For each case, the table gives the time required by the MISOCP solver
to find the optimal design; the (initial) lower bound on the
$D$-efficiency of the optimal design, compared to the constrained
approximate design;
the lower bound on the
$D$-efficiency of the optimal design
that is guaranteed after 10~min of computing time.}\label{figconcgraphs}
\end{figure}

Mixed integer optimization solvers rely on sophisticated branch-and-cut
algorithms. After each
iteration, the value $L=\Phi_D(M(\hat{\vec{w}}))$ of the best solution
$\hat{\vec{w}}$ found so far is compared to an upper bound $U$ provided
by a series of continuous relaxation of the problem, and the \emph{gap}
defined by $\delta=\frac{U-L}{L}$ is displayed.
Note that $\delta$ can directly be interpreted as a guarantee on the
$D$-efficiency of $\hat{\vec{w}}$, namely
$\operatorname{eff}_D(\hat{\vec{w}}) \geq(1+\delta)^{-1}$.
The following remark shows that for block designs,
the current best solution is actually proved to be exact $D$-optimal
as soon as the gap $\delta$ reaches a small tolerance parameter
$\varepsilon>0$.

\begin{remark} \label{remopt}
Let $T_{\vec{w}}$ denote the number of spanning trees
of the concurrence graph $G$ corresponding to an exact design $\vec{w}$,
and $T^*$ denote the maximal number of spanning trees for a particular
block design problem.
By using the fact that $T_{\vec{w}} = \det M(\vec{w})$ is an integer,
it can be seen that a tolerance parameter of
\[
\varepsilon= \biggl(1+\frac{1}{T^*} \biggr)^{{1}/{m}}-1 \simeq
\frac{1}{mT^*}
\]
ensures that the design $\vec{w}^*$ returned by the MISOCP approach is
(perfectly) optimal.
We have used this value of $\varepsilon$ in our
numerical experiments. When the value of $T^*$ is unknown, note that an
upper bound can be used
(e.g., the bound $T^* \leq\frac{1}{t}  (\frac{2N}{t-1}
)^{t-1}$ given by\vspace*{1pt} the
optimal design $\vec{w}=[\frac{N}{s},\ldots,\frac{N}{s}]^T$ for the
relaxed problem without integer constraints).
\end{remark}

To achieve a faster convergence,
a few variables can be set equal to $0$ or $1$ in order to break the
symmetry of the problem.
For example, if we search for a $D$-optimal design
in a class of exact designs with at least one treatment replicated
exactly $4$ times,
we can assume without loss of generality that
treatment~$1$ has replication number $4$, so
$w_{1,2}=w_{1,3}=w_{1,4}=w_{1,5}=1$ and
$w_{1,i}=0$ for all $i\in\{6,\ldots,t\}$.

The table in Figure~\ref{figconcgraphs}
gives information on the computing time required by CPLEX. In all four
situations,
the optimal design was found in the first seconds of computation. However,
note that the time required to obtain a certificate of optimality can
be much longer
[a few minutes for cases~(a) and~(d), and as much as 3~hours for case~(b)].
However, the bound on the $D$-efficiency provided by the MISOCP solver
after a few minutes
is already much better than the
standard bound of $D$-efficiency relative to the (constrained)
approximate optimal design.

This example also demonstrates that sometimes
we can use independent theoretical results to add some linear
constraints to the original optimum
design problem that can greatly\vadjust{\goodbreak} improve the computational efficiency.
Indeed, it has been conjectured that
every optimal block design with blocks of size two is (almost)
equireplicate for $t-1 \leq N \leq {t \choose 2}$. The conjecture is known
to hold for $t\leq11$~\cite{CE05}
and for all pairs $(t,N)$ such that $N\geq{t\choose 2}-t+2$~\cite{PBS98}.
The MISOCP solver required 333.7~s to obtain a certificate of optimality
of the
design plotted in Figure~\ref{figconcgraphs}(a)
in the class of equireplicate designs.
In contrast, several hours of computation are required
if we omit the constraints
on the replication numbers in the MISOCP formulation.

More computational results for optimal block designs can be found in
an earlier version of this manuscript that is available on the web~\cite
{SagnolH2013exactarxiv}.
In particular, we show that even for the case of standard
(unconstrained) exact design problems ($\mathcal{W}=\mathcal{W}_N$),
the MISOCP approach sometimes outperforms state-of-the-art algorithms
such as the $KL$-exchange procedure~\cite{AD92}.
The manuscript~\cite{SagnolH2013exactarxiv} also presents numerical
results on other criteria, such as $A$-optimality and $G$-optimality.

\subsection*{Locally $D$-optimal design in a study of chemical kinetics}

Another classical field of application of the theory of optimal
experimental designs is the study of chemical kinetics.
Here,
the goal is to select the points in time at which a chemical reaction
should be observed,
to estimate the kinetic parameters $\vecc{\theta}\in\R^m$ of the
reaction (rates, orders, etc.).
The measurements at time $t$ are of the form $\vec{y}_t = \vecc{\eta}_t(\vecc{\theta}) + \vecc{\varepsilon}_t$,
where $\vecc{\eta}_t(\vecc{\theta})=[\eta_t^1,\ldots,\eta_t^k]^T$ is the
vector of the concentrations of $k$ reactants at time $t$
and $\vecc{\varepsilon}_t$ is a random error.
The kinetic models are usually
given as a set of differential equations, which can be solved
numerically to find the concentrations $\vecc{\eta}_t(\vecc{\theta})$
over time. Unlike the linear model described in the introduction of
this paper, in chemical kinetics the
expected measurements $\mathbb{E}[\vec{y}_t] = \vecc{\eta}_t(\vecc{\theta
})$ at time $t$
depend nonlinearly on the vector $\vecc{\theta}$ of unknown parameters
of the reaction.
So a classical approach is to search for a \textit{locally optimal
design} using
a prior estimate $\vecc{\theta}_0$ of the parameter, that is,
a design which would be optimal if the true value of the parameters was
$\vecc{\theta}_0$. To do this, the
observation equations are linearized around $\vecc{\theta}_0$, so in practice
we replace the observation matrix $A_t$ of each individual trial at
time $t$ by its \emph{sensitivity} at $\vecc{\theta}_0$,
which is defined\vspace*{2pt} as
\[
\left.F_t:= \frac{\partial\vecc{\eta}_t(\vecc{\theta})}{\partial\vecc{\theta}} \bigg|_{\vecc{\theta} = \vecc{\theta}_0} = %
\pmatrix{
\displaystyle\frac{\partial\eta_t^1}{\partial\theta_1} & \cdots& \displaystyle\frac{\partial
\eta_t^k}{\partial\theta_1}
\cr
\vdots& \ddots& \vdots
\cr
\displaystyle\frac{\partial\eta_t^1}{\partial\theta_m} & \cdots& \displaystyle
\frac{\partial
\eta_t^k}{\partial\theta_m} } \right|_{\vecc{\theta} = \vecc{\theta}_0} \in\R^{m\times k}.\vspace*{2pt}
\]

A classical example is presented in~\cite{AD92}, the study of two
consecutive reactions
\[
A \mathop{\rightarrow}^{\theta_1} B \mathop{\rightarrow}^{\theta_2} C.
\]
The chemical reactions are assumed to be of order $\theta_3$ and $\theta
_4$, respectively,
so the concentrations of the reactants are determined by the
differential equations
\begin{eqnarray}
\nonumber
\frac{d[A]}{dt} &=& -\theta_1 [A]^{\theta_3},
\\
\label{equadifconc} \frac{d[B]}{dt} &=& \theta_1 [A]^{\theta_3} -
\theta_2 [B]^{\theta_4},
\\
\nonumber
\frac{d[C]}{dt}& =& \theta_2 [B]^{\theta_4},
\end{eqnarray}
together with the initial condition $([A],[B],[C])|_{t=0} = (1,0,0)$.
These equations can be
differentiated with respect to $\theta_1,\ldots,\theta_4$, which yields
another set of
differential equations that determines the elements $\frac{\partial\eta
_t^j}{\partial\theta_i}$ of the sensitivity\vspace*{2pt} matrices.

We now assume that measurements can be performed at each $t\in
\mathcal{X}=\{0.2,0.4, \ldots,19.8,20\}$,
where the time is expressed in seconds,
and that the observed quantities are the concentrations of the
reactants $A$ and $C$, that is, $k=2$ and
$\vecc{\eta}_t^T = ([A](t), [C](t))$. We have solved numerically the
differential equations governing the
entries of $(F_t)_{t\in\mathcal{X}}$ for $\vecc{\theta}_0 :=
[1,0.5,1,2]^T$. These sensitivities are plotted in Figure~\ref{figsensitivities}.

\begin{figure}[t]

\includegraphics{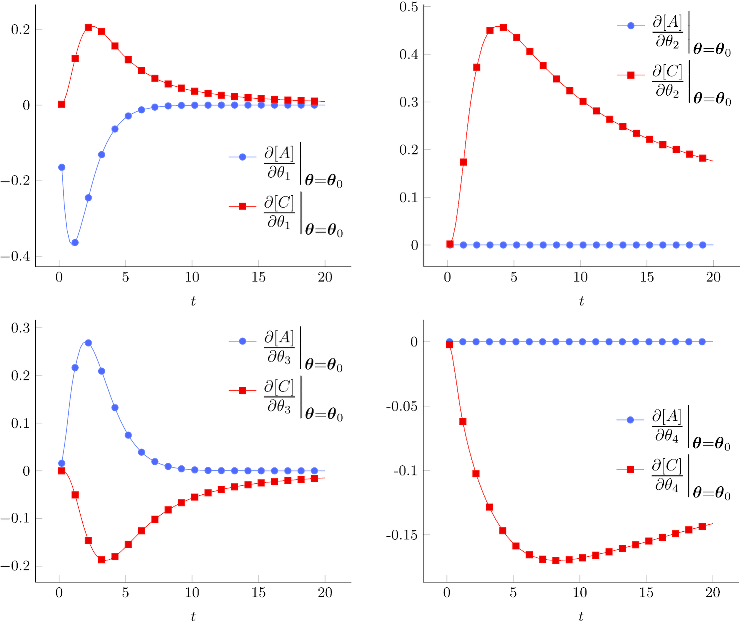}

\caption{Measurement sensitivities (entries of $F_t$) plotted against
time
for $\vecc{\theta}_0 = [1,0.5,1,2]^T$.}\label{figsensitivities}
\end{figure}
\begin{figure}

\includegraphics{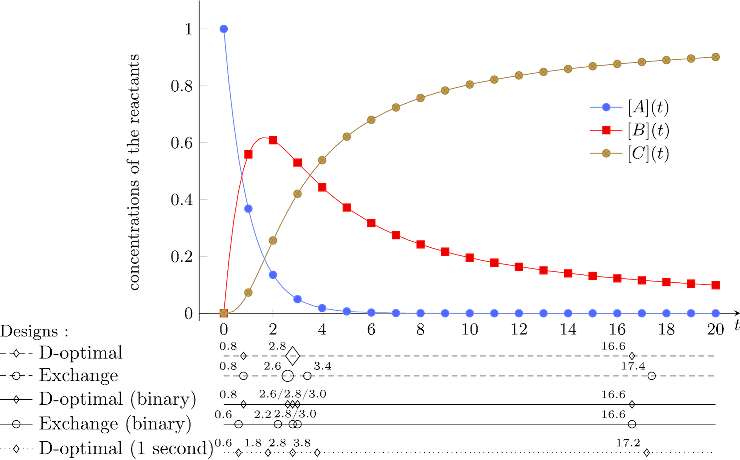}

\caption{Concentration\vspace*{1pt} of the reactants against time
(determined by solving equation~\protect\eqref{equadifconc},
assuming $\vecc{\theta}=\vecc{\theta}_0=[1,0.5,1,2]^T$).
Several designs are represented below the graph. The marks indicate the
time at which the measurements should be performed,
and the size of the marks indicate the number of measurements at a
given point in time. \textit{Binary} means that the
design space is restricted to designs having at most one measurement
for each $t\in\mathcal{X}$, and 1 \textit{second} means that
at least 1 second must separate 2 measurements.}\label{figabc}
\end{figure}

We used the MISOCP method to compute the exact $D$-optimal design of
size $N=5$ for this problem (for the prior estimate $\vecc{\theta}_0$).
The optimum consists in taking 1 measurement at $t=0.8$, 3 measurements
at $t=2.8$ and 1 measurement at $t=16.6$.
In comparison, the exchange algorithm (using the same settings as
described for the block designs, with $N^{R}=100$)
found a design with 1 measurement for each $t\in\{0.8,3.4,17.4\}$ and 2
measurements at $t=2.6$. This design is of course
very close to the optimum (its $D$-efficiency is $98.42\%$), but we
note that the \textit{true} optimum could
not be identified by the
exchange algorithm, even with a very large number of tries. We ran the
exchange procedure $N^{R}=5000$ times
which took 100 s and returned a design of $D$-efficiency $99.42\%$,
while the MISOCP found a provable optimal design after $25$ s
(CPLEX returned the status \texttt{MIP\_OPTIMAL}).

We plotted these designs in Figure~\ref{figabc} together with the
concentrations of the reactants over time when we
assume $\vecc{\theta}=\vecc{\theta}_0$. In the figure, we have also
plotted other designs which can be of interest to practitioners.
For example, it might be natural to search designs where at most 1
measurement is taken at a given point in time.
The exchange algorithm can also be adapted to the case of binary designs
(by rejecting candidate points that already support the design during
the exchange procedure).
It returned a design of $D$-efficiency $98.97\%$. The last case we have
considered is the following: assume that the
experimenter must wait at least one second after a measurement before
performing another measurement. This
constraint can be modeled as\vadjust{\goodbreak} a set of inequalities that can be added
into the MISOCP formulation,
\begin{eqnarray*}
&& \{w_{0.2} + w_{0.4} + w_{0.6} +
w_{0.8} +w_{1.0} \leq1,\\[-3pt]
&&\hspace*{6pt} w_{0.4} + w_{0.6}
+ w_{0.8} + w_{1.0} +w_{1.2} \leq1,\ldots,
\\[-3pt]
&&\hspace*{6pt} w_{19.2} + w_{19.4} + w_{19.6} +
w_{19.8} +w_{20.0} \leq 1\}.
\end{eqnarray*}
This model was solved in 42 s with CPLEX,
and the corresponding optimal design is depicted on the last row of
Figure~\ref{figabc}.
We do not know of any other algorithm that can handle this type of
exact design problem with several linear\vadjust{\goodbreak} constraints.

\begin{appendix}
\section*{Appendix: Other optimality criteria} \label{secothercrits}
\subsection{$A_K$-optimality} \label{secAk}
Another widely used criterion in optimal design is $A$-optimality, which
is defined by
\[
\Phi_A: M \to \cases{ \bigl(\operatorname{trace} M^{-1}
\bigr)^{-1}, &\quad$\mbox{if } M\mbox{ is nonsingular}$;
\vspace*{3pt}\cr
0, &\quad$
\mbox{otherwise}$.}
\]
More generally, it is possible to use the criterion of $A_K$-optimality
if the experimenter is
interested in the estimation of the parameter subsystem $\vecc{\vartheta}=K^T \vecc{\theta}$,
\[
\Phi_{A|K}: M \to \cases{ %
 \bigl(\operatorname{trace}
K^T M^{-} K \bigr)^{-1}, & \quad$\mbox{if }
\range K \subseteq\range M$;
\vspace*{3pt}\cr
0, &$\quad\mbox{otherwise}$.} %
\]

Here $M^-$ denotes a generalized inverse of $M$; see the discussion
following equation~\eqref{phiDK} in the \hyperref[sec1]{Introduction}.
Note that $\Phi_{A|K}$ coincides with $\Phi_A$ if $K=\vec{I}_m$,
and $\Phi_{A|K}$ reduces to the criterion of $\vec{c}$-optimality when
$K=\vec{c}\neq\vec{0}$ is a column vector.

The following lemma was already used in~\cite{Sagnol09SOCP},
under a slightly different form.
In fact, this lemma is a consequence of the Gauss--Markov theorem,
which states that
the variance--covariance matrix of the best linear unbiased estimator
of $K^T \vecc{\theta}$ is proportional to $K^T M(\vec{w})^- K$
(e.g., Pukelsheim~\cite{Puk93}).

\setcounter{theorem}{0}
\begin{lemma} \label{lemG-M}
Let $K$ be an $(m\times k)$-matrix,
and let $\vec{w}\in\R_+^s$ be a vector of design weights,
such that the estimability condition
$\range K \subseteq\range M(\vec{w})$
is satisfied. Define $I:=\{i\in[s]: w_i>0\}$. Then
\setcounter{equation}{0}
\begin{eqnarray}
\operatorname{trace} K^T M(\vec{w})^- K &=& \min
_{(Z_i)_{i \in I}} \sum_{i \in I} \quad
\frac{\Vert Z_i\Vert_F^2}{w_i}
\nonumber
\\[-8pt]
\label{G-Mineq}
\\[-8pt]
\nonumber
&&\qquad\mbox{s.t.}\qquad\sum_{i \in I}
A_i Z_i = K,
\nonumber
\end{eqnarray}
where the variables $Z_i$ ($i\in I$) are of size $\ell_i \times k$.
\end{lemma}

After some changes of variable, we obtain an SOC representation of $\Phi_{A|K}$:
\begin{proposition} \label{propSOCP-A}
Let $K$ be an $(m\times k)$-matrix, and let $\vec{w}\in\R_+^s$ be a
vector of design weights.
Then
\begin{eqnarray}
\Phi_{A|K} \bigl( M(\vec{w}) \bigr) &=& \max_{\vecc{\mu}\in\R_+^s, Y_i \in\R
^{\ell_i \times k}}
\sum_{i\in[s]} \mu_i
\nonumber
\\
\label{SOCP-A} &&\qquad\mbox{s.t.}\qquad\sum_{i\in[s]}
A_i Y_i = \biggl(\sum_{i\in[s]}
\mu_i\biggr) K,
\\
&&\qquad\qquad\hspace*{15pt}\forall i\in[s],\Vert Y_i \Vert_F^2
\leq w_i \mu _i.
\nonumber
\end{eqnarray}
\end{proposition}

\begin{pf}
We first handle the case where the estimability condition is not satisfied.
In this situation,
we have $\Phi_{A|K} ( M(\vec{w})  ) = 0$,
and we will see that the first constraint of problem~\eqref{SOCP-A}
can only be satisfied if $\sum_{i=1}^s \mu_i = 0$. Note that the second
constraint of problem~\eqref{SOCP-A} implies $Y_i=0$ for all $i\notin I$.
Hence every column of the matrix $\sum_{i=1}^s A_i Y_i$ must be in the
set $\range[\sqrt{w_1} A_{1},\ldots, \sqrt{w_s} A_{s}] = \range M(\vec{w})$.
Thus if (at least) one column of $K$ is
not included in the range of $M(\vec{w})$,
then we must have $\sum_i \mu_i = 0$.

Now, assume that the estimability condition
$\range K \subseteq\range M(\vec{w})$
holds, so that
\[
\Phi_{A|K} \bigl( M(\vec{w}) \bigr) = \bigl(\operatorname{trace}
K^T M(\vec {w})^{-} K \bigr)^{-1}>0.
\]
Let $Z_i$ ($\forall i \in I$) be optimal matrices\vspace*{1.5pt}
for the problem on the right-hand side of~\eqref{G-Mineq}.
Then for all $i\in I$, define
$\lambda_i := \operatorname{trace} w_i^{-1} Z_i^T Z_i= w_i^{-1} \Vert
Z_i \Vert_F^2$,
$\mu_i :=  (\sum_j \lambda_j )^{-2}\lambda_i$, and
$Y_i :=  (\sum_j \lambda_j )^{-1} Z_i$
[note that $\sum_{i\in I} \lambda_i = \operatorname{trace} K^T M(\vec
{w})^{-} K > 0$],
and for $i\in[s]\setminus I$ let $\mu_i:=0$ and $Y_i := 0 \in\R^{\ell
_i \times k}$.
We have\vspace*{1pt} $\sum_{i\in[s]} \mu_i = (\sum_{i\in I} \lambda_i)^{-1} = \Phi
_{A|K} ( M(\vec{w})  )$,
and by construction the variables $\mu_i$ and $Y_i$ satisfy the
constraints of problem~\eqref{SOCP-A}.

Conversely, let $\mu_i$ and $Y_i$ be feasible variables for
problem~\eqref{SOCP-A}.
If $\sum_i \mu_i=0$, then we have $\sum_i \mu_i < \Phi_{A|K} ( M(\vec
{w})  )$.
Otherwise,\vspace*{1.5pt} define $Z_i := (\sum_{i\in[s]} \mu_i)^{-1} Y_i$,
so that the variables $Z_i$ ($i\in I$) are feasible for the problem on
the right-hand side of~\eqref{G-Mineq}.
Hence
\begin{eqnarray*}
\operatorname{trace} K^T M(\vec{w})^- K &\leq & \sum
_{i\in I} \frac
{1}{w_i} \Vert Z_i
\Vert_F^2 = \sum_{i\in I}
\frac{1}{w_i (\sum_{i\in[s]} \mu_i)^2} \Vert Y_i \Vert _F^2
\\
&\leq & \sum_{i\in I} \frac{w_i \mu_i}{w_i (\sum_{i\in[s]} \mu_i)^2} \leq
\frac{1}{\sum_{i\in[s]} \mu_i}.
\end{eqnarray*}

Finally, we obtain the desired inequality by taking the inverse
\[
\sum_{i\in[s]} \mu_i \leq\Phi_{A|K} \bigl( M(\vec{w})  \bigr).
\]
This completes the proof of the proposition.
\end{pf}

\begin{corollary} \label{coroAK}
$\!$Let $K$ be an $m \times k$ matrix. The function$\vec{w} \mapsto\break
\Phi_{A|K} (M(\vec{w}) )$
is SOC-representable.
\end{corollary}

The reformulation of problem~\eqref{theBasicProblem} for the criterion
$\Phi=\Phi_{A|K}$
as an (MI)SOCP is indicated in Table~\ref{tabformulations}.

\begin{remark}[(The case of $\vec{c}$-optimality)]
The case of $\vec{c}$-optimality arises as a special case of both $A_K$
and $D_K$-optimality
when the matrix $K=\vec{c}\neq\vec{0}$ is a column vector ($k=1$). 
The two SOCP formulations (for $\Phi_{A|\vec{c}}$ and $\Phi_{D|\vec{c}}$- in Table~\ref{tabformulations})
are equivalent, which can be verified by the change of variables
$Y_i = J_{1,1}^{-1} Z_i$, $\mu_i = J_{1,1}^{-2} t_{i1}$. (Note that
here the matrix $J$ is of size $1\times 1$, i.e., a scalar.)
\end{remark}

We next show how Proposition~\ref{propSOCP-A}
can be used to obtain an SOC representation of $G$ and $I$-optimality.

\subsection{$G$-optimality}

A criterion closely related to $D$-optimality is the criterion of
$G$-optimality,
\[
\Phi_G: M \to \Bigl(\max_{i\in[s]}
\operatorname{trace} A_i^T M^- A_i
\Bigr)^{-1} = \min_{i\in[s]} \Phi_{A|A_i}(M),
\]
where the equality holds if we use the convention $\operatorname{trace}
K^T M^- K:=+\infty$
for all matrices $M$ that do not satisfy the estimability condition
($\operatorname{range} K \nsubseteq\operatorname{range} M$).
In the common case of single-response experiments for linear models,
the matrices $A_i$ are column vectors, and the scalar $\sigma^2 A_i^T
M(\vec{w})^- A_i$ represents the
variance of the prediction $\hat{\vec{y}}_i=A_i^T \hat{\vecc{\theta}}$.
Hence $G$-optimal designs
minimize the maximum variance of the predicted values $\hat{\vec{y}}_1,
\ldots, \hat{\vec{y}}_s$.

The $G$ and $D$-optimality criteria are related to each other by the celebrated
equivalence theorem of Kiefer and Wolfowitz~\cite{KW60}, which was generalized
to the case of multivariate regression ($\ell_i>1$) by Fedorov in
1972~\cite{Fed72}.
An important consequence of this theorem is that $D$- and $G$-optimal
designs coincide when the weight domain $\mathcal{W}$ is the standard
probability simplex $\mathcal{W}_\Delta$. However, exact $G$-optimal
designs do not necessarily coincide with their $D$-optimal
counterparts.
In a recent article~\cite{RJBM10}, the Brent minimization algorithm was proposed
to compute near exact $G$-optimal factorial designs. But in general,
we do not know any standard algorithm for the computation of exact
$G$-optimal designs
or $G$-optimal
designs over arbitrary weight domains $\mathcal{W}$
that are defined by a set of linear inequalities.

We know from Corollary~\ref{coroAK} that the concave functions $f_i:
\vec{w}\to\break \Phi_{A|A_i}  ( M(\vec{w})  )$
are SOC-representable, and hence their minimum is also concave and
SOC-representable.
An (MI)SOCP formulation of problem~\eqref{theBasicProblem} for the
criterion $\Phi= \Phi_G$ is
indicated in Table~\ref{tabformulations}. For the case where the
weight domain $\mathcal{W}$
is the probability simplex $\mathcal{W}_\Delta$, it gives a new
alternative SOCP
formulation for $D$-optimality. Note, however, that in this situation,
the SOCP formulation~\eqref{Dopt-oldSOCP}
for $D$-optimality from~\cite{Sagnol09SOCP} is usually more compact
(i.e., it involves fewer
variables and fewer constraints) than the $G$-optimality SOCP of
Table~\ref{tabformulations}.

\subsection{$I$-optimality} \label{secIopt}

Another widely used criterion is the one of $I$-optimality (or
$V$-optimality). Here, the criterion is the
inverse of the average of the variances of the predicted values $\hat
{\vec{y}}_1, \ldots, \hat{\vec{y}}_s$:
%
\[
\Phi_I: M \to \biggl(\frac{1}{s} \sum
_{i\in[s]} \operatorname{trace} A_i^T M^-
A_i \biggr)^{-1}.
\]

In fact, this criterion coincides with the $\Phi_{A|K}$ criterion, by
setting $K$
to any matrix of full column rank satisfying $KK^T = \frac{1}{s} \sum_{i=1}^s A_i A_i^T$; see, for example, Section~9.8 in~\cite{Puk93}.
Hence $\Phi_I$-optimal designs can be computed by SOCP.
Note that there is also a weighted version of $I$-optimality,
which can be reduced to an $A_K$-optimal design problem in the same manner.

\subsection{Bayesian optimal designs for nonlinear models} \label{secBopt}

For nonlinear models, the information matrix of a design $\vec{w}$
depends on the
value of the unknown parameter $\vecc{\theta}$ [we denote it by $M(\vec{w},\vecc{\theta})$];
see, for example,~\cite{CV95}. One way to handle this
challenging cyclic problem is to search  a design $\vec{w}$ maximizing the
expected value $\Phi_{\pi}(\vec{w})$ of the criterion $\Phi$ with
respect to some prior $\pi$,
\[
\Phi_{\pi}(\vec{w}):= \int_{\vecc{\theta}\in\R^m} \Phi \bigl( M(
\vec {w},\vecc{\theta}) \bigr)\, \pi(d\vecc{\theta}).
\]
Another alternative, known as \emph{standardized Bayesian design}, is
to search for a design maximizing the
expected efficiency
\[
\phi_{\pi}(\vec{w}):= \int_{\vecc{\theta}\in\R^m}
\frac{\Phi ( M(\vec
{w},\vecc{\theta})  )}{\max_{\vecc{\omega} \in\mathcal{W}}\, \Phi
( M(\vecc{\omega},\vecc{\theta})  )}\, \pi(d\vecc{\theta}).
\]

In a recent article, Duarte and Wong approximated such integrals by
finite sums
using Gaussian quadrature formulas~\cite{DW14}, in order to obtain SDP
formulations of Bayesian optimal design problems.
By using the same technique, we immediately see that the Bayesian
versions $\Phi_\pi$ and $\phi_\pi$ of
a SOC-representable criterion $\Phi$ are also SOC-representable (modulo
the approximation of the integral by a finite sum).
This offers the possibility of computing (constrained) exact Bayesian
designs by using MISOCP solvers.

Finally, we point out that the standard Bayesian versions of the $D$-
and $A$-criteria have forms that slightly
differ from the formulas given above,
and which have other statistical interpretations (see~\cite{CV95} for
more details),
\begin{eqnarray*}
\Phi_{D,\pi}(\vec{w})&:=& \int_{\vecc{\theta}\in\R^m} \log\!\det M(
\vec {w},\vecc{\theta}) \pi(d\vecc{\theta}),
\\
\Phi_{A,\pi}(\vec{w})&:=& - \int_{\vecc{\theta}\in\R^m} \operatorname
{trace} M(\vec{w},\vecc{\theta})^{-1} \pi(d\vecc{\theta}).
\end{eqnarray*}

Bayesian optimality with respect to the above criteria can also be formulated
as an (MI)SOCP, by combining the techniques used in the present paper
with those of~\cite{DW14}.
\end{appendix}

\printaddresses
\end{document}